\documentclass[12pt]{article}
\usepackage{amsmath}
\usepackage{amsfonts}
\usepackage{amssymb}
\usepackage{amsthm}
\usepackage{stackrel}

\usepackage{color}

\usepackage{figsize}

\usepackage{graphics}

\newtheorem{df}{Definition}[section]
\newtheorem{thm}{Theorem}[section]

\newcommand{\R}{{\rm I}\kern-0.18em{\rm R}}
\newcommand{\1}{{\rm 1}\kern-0.25em{\rm I}}
\newcommand{\E}{{\rm I}\kern-0.18em{\rm E}}
\newcommand{\p}{{\rm I}\kern-0.18em{\rm P}}
\makeatletter

\title{Outliers and related problems}
\author{Lev B. Klebanov\footnote{Department of Probability and Mathematical Statistics, MFF, Charles University, Czech Republic. e-mail: lev.klebanov@mff.cuni.cz }, Jaromir Antoch\footnote{Department of Probability and Mathematical Statistics, MFF, Charles University, Czech Republic.}, Andrea Karlova\footnote{Institute of Information Theory and Automation, CAS, Prague, Czech Republic.}\\ and Ashot V. Kakosyan\footnote{Yerevan State University, Yerevan, Armenia.}}
\date{}
\begin{document}
\maketitle
\begin{abstract}
We define outliers as a set of observations which contradicts the proposed mathematical (statistical) model and  
we discuss the frequently observed types of the outliers. Further we explore what changes in the model have to be made 
in order to avoid the occurance of the outliers. We observe that some variants of the outliers lead to classical results in probability, 
such as the law of large numbers and the concept of heavy tailed distributions.

\end{abstract}

\section{Introduction and suggestive reflections}\label{s1}
\setcounter{equation}{0}
In this paper we revise the concept of the {\bf outliers}. We found the contemporary notion rather vague, 
which motivates us to carefuly dispute its meaning. Let us start by closely looking at the definion of outlier 
provided by the widely popular free internet encyclopedia Wikipedia. The outliers are defined there as follows: 
``In statistics, an outlier is an observation point that is distant from other observations. 
An outlier may be due to variability in the measurement or it may indicate experimental error; 
the latter are sometimes excluded from the data." Obviously, the definition is given neither in mathematically nor statistically correct way. 
In particular, we found the description of "the point being distant" from other observations rather confusing.\footnote{A little bit better seems to be a definition given on NISTA site:``An outlier is an observation that lies an abnormal distance from other values in a random sample from a population. In a sense, this definition leaves it up to the analyst (or a consensus process) to decide what will be considered abnormal. Before abnormal observations can be singled out, it is necessary to characterize normal observations." However, it has similar drawbacks.} In our opinion, it is essential to specify 
some measurement unit of the considered distance and mainly the definition of the corresponding considered distance. 
Therefore we wish to conclude that the term outlier in such a setup is highly depended on the choice of topology 
and geometry of the space in which we consider our experiment. In the same manner, we found the term "experimental error" 
equally misleading. Say, outlier is an observation which is not connected to the particular experiment, 
and so this observation will not appear in the next experiment. 
However, the statistics is devoted to repeating the experiments, 
and such observations will be automatically excluded from further experiments and study. 
Now consider the  possibility that such "distant" observations remain appearing in the repetitions of our experimental study. 
In that case,  we need to keep the observations attributed to the experiment. Therefore, it is misleading to label the observations as "errors".
For example, the trigerring event of occurance of such observations can be caused by  the design of the particular experiment, 
i.e. the way how the experiment is designed does not capture the nature of  corresponding applied problem. 
As a result, some observations may appear as a natural phenomena seamlessly to the considered problem. 
However, there are no mathematical or statistical tools to recognize such a situation and so we are left with concluding that:
{\it such observations are in contradiction with mathematical model choosen to describe the practical model under study}. 
Of course, if some observations are in contradiction with one model, they may be in a good agreement with another model. 
And so we conclude that the notion of outliers is a model sensitive, i.e. the outlier needs to be associated with 
the concrete mathematical or statistical model. 

Based on our initial discussion, let us give the following definition.
\begin{df}\label{de1}
Consider a mathematical model of some real phenomena experiment. 
We say that an observation is the outlier for this particular model 
if it is "in contradiction" with the model, i.e. it is either impossible to have such an observation under the assumption that the model holds, 
or the probability to obtain such observation for the case of true model is extremely low. 
If the probability is very small yet non-zero, we denoted the probability as $\beta$, we will call 
relevant observation the $\beta$-outlier. 
\end{df}

Definition \ref{de1} gives precise sense to the second part of the Wikipedia definition. 
However, it provides no connection to the first part. In the following sections
of this paper we provide the arguments and explanations that some typical cases of the outliers appearance 
in the statistical modelling are closely connected with the  properly defined "the distant character" of them. 
These "proper definitions" provide meaningful suggestions to posssible model modifications 
in order to include the outliers as an element of the new model. 
Note that some ideas of the modification of outliers definitions were already considered in \cite{Kl16}.
 
\section{First definition of distant outliers}\label{s2}
\setcounter{equation}{0}
\subsection{Outliers of the first kind}\label{2ss1}

In this section we explore the situation when some observations 
observation are "distant" from the others. What is the "unit of measurement" for such a distance?
The natural way to start is to measure the distance of the observations to their mean value 
in terms of sample variance. 

Suppose that $X_1,X_2, \ldots ,X_n$ is a sequence of independent identically distributed (i.i.d.) random variables. Denote by
\[ \bar{x}_n=\frac{1}{n}\sum_{j=1}^{n}X_j, \;\; s^2_n=\frac{1}{n}\sum_{j=1}^{n}(X_j-\bar{x})^2\]
their empirical mean and empirical variance correspondingly. Let $k>0$ be a fixed number. 
Namely, let us estimate the following probability 
\begin{equation}\label{eq3}
p_n=\p \{|X-\bar{x}_n|/s_n>k\},
\end{equation}
\begin{df}\label{de2}
We say that the distribution of $X$ produces outliers of the first kind if the probability (\ref{eq3}) is high (say, higher than for normal distribution). 
\end{df}
Really, if one has a model based on Gaussian distribution then the presence of many observations with $p_n$ greater that for normal case contradicts to the model, and the observations appears to be outliers in the sense of our Definition \ref{de1}. Such approach was used in financial mathematics to show the Gaussian distribution provides bad model for corresponding data (see, for example, \cite{EK, BMW}).

The observations $X_j$ for which the inequality $|X_j-\bar{x}_n|/s_n>k$ holds appears to be outliers for Gaussian model. In some financial models the presence of them were considered as an argument for the existence of heavy tails for real distributions.  Unfortunately, this is not so (see \cite{KV, KlF, KTK}).

\begin{thm}\label{th1}(see \cite{KTK}) Suppose that $X_1,X_2, \ldots ,X_n$ is a sequence of i.i.d. r.v.s belonging to a domain of attraction of strictly stable random variable with index of stability $\alpha \in (0,2)$. Then
		\begin{equation}
		\label{eq4}
		\lim_{n \to \infty}p_n =0.
		\end{equation}

\end{thm} 
\begin{proof}
	Since $X_j,\; j=1, \ldots ,n$ belong to the domain of attraction of strictly stable random variable with index $\alpha <2$, it is also true that $X_1^2, \ldots , X_n^2$ belong to the domain of attraction of one-sided stable distribution with index $\alpha /2$. 
	
	1) Consider at first the case $1<\alpha <2$. In this case, $\bar{x}_n \stackrel[n \to \infty]{}{\longrightarrow }a=\E X_1$ and
	$s_n \stackrel[n \to \infty]{}{\longrightarrow }\infty$. We have
	\[ \p\{|X_1 - \bar{x}_n| >k s_n\} = \p\{X_1> ks_n+\bar{x}_n\}+\p\{X_1< -ks_n+\bar{x}_n\} = \]
	\[=\p\{X_1> ks_n+a+o(1)\}+\p\{X_1< -ks_n+a +o(1)\}  \stackrel[n \to \infty]{}{\longrightarrow } 0. \]
	
	2) Suppose now that $0<\alpha < 1$. In this case, we have $\bar{x}_n \sim n^{1/\alpha -1}Y$ as $n \to \infty$. Here $Y$ is $\alpha$-stable random variable, and the sign $\sim$ is used for asymptotic equivalence. Similarly, 
	\[ s^2_n =\frac{1}{n}\sum_{j=1}^{n}X_j^2 - \bar{x}^2_n \sim n^{2/\alpha -1}Z (1+o(1)), \]
	where $Z$ has one-sided positive stable distribution with index $\alpha /2$. We have
	\[ \p\{|X_1-\bar{x}_n|>k s_n\}=\p\{ (X_1-\bar{x}_n)^2>k s^2_n\}=\]
	\[=\p\{X_1^2 > n^{2/\alpha -1}Z (1+o(1))\}  \stackrel[n \to \infty]{}{\longrightarrow } 0. \] 
	
	3) In the case $\alpha =1$ we deal with Cauchy distribution. The proof for this case is very similar to that in the case 2). We omit the details.
\end{proof} 

From this Theorem it follows that (for sufficiently large $n$) many heavy-tailed distributions will not produce any outliers of the first kind. Moreover, now we see the the presence of outliers of the first kind is in contradiction with many models having heavy tailed distributions, particularly, with models involved stable distributions.
By the way, word {\bf variability} is not defined precisely, too. It shows, that high variability may denote something different than high standard deviation. We will discuss this in Section \ref{s3}, but now let us continue the study of distributions with high probability $p_n$.

\subsection{How to obtain more outliers of the first kind?}\label{2ss2}

Here we discuss a way of constructing from a distribution another one having a higher probability to observe outliers. We call this procedure "put tail down".

\vspace{0.2cm}
Let $F(x)$ be a probability distribution function of random variable $X$ having finite second moment $\sigma^2$ and such that $F(-x) = 1-F(x)$ for all $x \in \R^1$. Take a parameter $p \in (0,1)$ and fix it. Define a new function 
\[ F_p(x)=(1-p)F(x) + p H(x),\]
where $H(x) = 0$ for $x<0$, and $H(x)=1$ for $x>0$. It is clear that $F_p(x)$ is probability distribution function for any $p \in (0,1)$. Of course, $F_p$ also has finite second moment $\sigma_p^2$, and $F_p(-x)=1-F_p(x)$. However, $\sigma_p^2=(1-p)\sigma^2$. Let $Y_p$ be a random variable with probability distribution function $F_p$. Then
\[ \p\{|Y_p|>k\sqrt{1-p}\sigma\} = 2 \p\{Y_p>k\sqrt{1-p}\sigma\}=\] \[=2(1-p)\bigl(1-F(k\sqrt{1-p}\sigma)\bigr).\]

Denoting $\bar{F}(x)=1-F(x)$ rewrite previous equality in the form
\begin{equation}
\label{eq3a}
\p\{|Y_p|>k\sqrt{1-p}\sigma\} = 2 (1-p) \bar{F}(k\sqrt{1-p}\sigma).
\end{equation}
For $Y_p$ to have more outliers than $X$ it is sufficient that
\begin{equation}\label{eq4a}
(1-p) \bar{F}(k\sqrt{1-p}\sigma) > \bar{F}(k\sigma).
\end{equation}
There are many cases in which inequality (\ref{eq4a}) is true for sufficiently large values of $k$. Let us mention two of them.
\begin{enumerate}
	\item Random variable $X$ has exponential tail. More precisely, 
	\[ \bar{F}(x) \sim C e^{-a x}, \;\text{as}\; x \to \infty , \]
	for some positive constants $C$ and $a$. In this case, inequality (\ref{eq4a}) is equivalent for sufficiently large $k$ to
	\[(1-p) > Exp\{-a\cdot k \cdot \sigma \cdot (1-\sqrt{1-p})\},\]
	which is obviously true for large $k$.
	\item $F$ has power tail, that is $\bar{F}(x) \sim C/x^{\alpha}$, where $\alpha>2$ in view of existence of finite second moment. Simple calculations show that (\ref{eq4a}) is equivalent as $k \to \infty $ to
	\[ (1-p)^{1-\alpha/2} <1. \]
\end{enumerate}

The last inequality is true for $\alpha >2$. 

\vspace{0.2cm}
Let us note that the function $F_p$ has a jump at zero. However, one can obtain similar effect without such jump by using a smoothing procedure, that is by approximating $F_p$ by smooth functions.

"Put tail down" procedure allows us to obtain more outliers in view 
of two its elements. First element consists in changing the tail by smaller, but proportional to previous with coefficient $1-p$. The second element consist in moving a part of mass into origin (or into a small neighborhood of it), which reduces the variance.

\vspace{0.2cm} The procedure described above shows us that the presence of outliers may have no connection with existence of heavy tails of underlying distribution or with experimental errors. 

\subsection{On extremal and related distributions with outliers of the first kind}\label{2ss3}

In the case of finite variance it is possible to find a distribution maximizing the probability
\begin{equation}\label{eq5}
p(k)= \p\{|X|>k \sigma\}
\end{equation}
for the case $\E X =0$, $k>1$. Corresponding boundary is given by Selberg inequality (see, for example, \cite{KS}). Namely, if $X$ is a random variable such that 
\[ \E X =0, \;\; \E X^2 =\sigma^2 < \infty\]
then for any $k>1$
\begin{equation}\label{eq6}
p(k) \leq \frac{1}{k^2}.
\end{equation}
The equality in (\ref{eq6}) is attended on a distribution concentrated at 3 points: $-\sigma$, $0$ and $\sigma$. 

In the case of $k=3$ the boundary in (\ref{eq6}) is $1/k^2 = 1/9$, which shows that for extremal distribution one may have many outliers of the first kind. However, corresponding distribution has a compact support.

It is clear that extremal distribution has very specific form and rarely appears in applications. Therefore, it is of essential interest to find out which properties of a distribution lead to the presence of rather high probability for outliers of the first kind. The form of extremal distribution tells us that a part of it has to be concentrated near mean value, while other part must be not too close to the mean. Suppose that $X_1, X_2, \ldots ,X_n$ are i.i.d. random variables. Denote by $0 \leq X_{(1)}\leq X_{(2)} \leq \ldots \leq X_{(n)}$ ordered values of absolute values of observations $|X_{j}|$, $j=1, \ldots ,n$. It seems to be true that to get many outliers of the first kind one needs to have rather high probability of the event $\p\{X_{(2)}>\rho X_{(1)}\}$, where $\rho >1$. Let us verify this statement.

To this aim calculate the probability of $X_{(2)}$ to be greater than $\rho X_{(1)}$ for a fixed $\rho >1$. Suppose that $X_1, \ldots ,X_n$ are i.i.d. random variables, and $X_{(1)}, \ldots , X_{(n)}$ are the observations ordered in its absolute values. Suppose that random variable $|X_1|$ has absolute continuous distribution function $F(x)$, and $p(x)$ is its density. Denote by $p_{1,2}(x,y)$ the common density of $X_{(1)}$ and $X_{(2)}$. We have (see, for example, \cite{Dav}) 
\begin{equation}\label{eq7}
p_{1,2}(x,y) = n(n-1) F^{n-2}(y)p(x) p(y),
\end{equation}
for $x \leq y$. Therefore the probability of the event that $X_{(2)} \geq \rho X_{(1)}$ is
\begin{equation}\label{eq8}
\begin{aligned}
\p\{X_{(2)} \geq \rho X_{(1)}\} = \\ =n \int_{0}^{\infty}\Bigl(\bigl(1-F(x)\bigr)^{n-1}-\bigl(1-F(\rho x)\bigr)^{n-1}\Bigr)p(x)dx= \\
=1-n \int_{0}^{\infty}\left(1-F(\rho x)\right)^{n-1}p(x) dx.
\end{aligned}
\end{equation}
Let us try to study limit behavior of the probability (\ref{eq8}) for large values of sample size $n$. We have
\begin{equation}\label{eq9}
\p\{X_{(2)} \geq \rho X_{(1)}\} =1-n \int_{0}^{\infty}\left(1-F(\rho x)\right)^{n-1}\rho p(\rho x)\frac{p(x)}{\rho p(\rho x)} dx. 
\end{equation}
Assume that 
\begin{equation}\label{eq10}
\lim_{x \to \infty}\bigl(1-F(\rho x)\bigr)^{n}\frac{p(x)}{\rho p(\rho x)} =0. 
\end{equation}
Integrating by parts in (\ref{eq9}) gives us
\begin{equation}\label{eq11}
\p\{X_{(2)} \geq \rho X_{(1)}\}= \lim_{x \to 0}\frac{p(x)}{\rho p(\rho x)} -\int_{0}^{\infty}\left(F(\rho x)\right)^{n} \frac{d}{d x}\Bigl(\frac{p(x)}{\rho p(\rho x)}\Bigr) dx
\end{equation}
If the function
\[ \frac{d}{d x}\frac{p(x)}{\rho p(\rho x)} \]
is absolute integrable over $(0,\infty)$ then 
\[\int_{0}^{\infty}\left(F(\rho x)\right)^{n} \frac{d}{d x}\frac{p(x)}{\rho p(\rho x)} dx \to 0 \]
as $n \to \infty$.
Therefore,
\begin{equation}\label{eq12}
\lim_{n \to \infty}\p\{X_{(2)} \geq \rho X_{(1)}\}= \lim_{x \to 0}\frac{p(x)}{\rho p(\rho x)},
\end{equation}
assuming that the limit in right-hand side of (\ref{eq12}) exists.

Finally, we obtain the following result.
\begin{thm}\label{th2}
Suppose that $X_1, \ldots ,X_n$ are i.i.d. random variables, and $X_{(1)}, \ldots , X_{(n)}$ are the observations ordered in its absolute values. Let random variable $|X_1|$ has absolute continuous distribution function $F(x)$, and let $p(x)$ be its density. Suppose that $p(x)$ is regularly varying function of index $\alpha -1$ at zero, the function 
\[  \frac{d}{d x}\frac{p(x)}{\rho p(\rho x)}\]
exists and is integrable over $(0,\infty)$, and 
\[ \lim_{x \to \infty}\bigl(1-F(\rho x)\bigr)^{n}\frac{p(x)}{\rho p(\rho x)} =0.  \]
Then
\begin{equation}\label{eq7}
\lim_{n \to \infty}\p\{X_{(2)} \leq \rho X_{(1)}\}=1-\frac{1}{\rho^{\alpha}}.
\end{equation}
\end{thm}
\begin{proof}
The statement of the Theorem follows from considerations given above and from the definition of regularly varying function (see, for example \cite{Sen}).
\end{proof}

Let us consider the probability $\p\{X_{(2)} \leq \rho X_{(1)}\}$ as a function of $\rho$. Under conditions of Theorem \ref{th2} this probability represents cumulative distribution function of Pareto law with parameter $\alpha$. Is it possible to find a distribution function for which the equality holds not only in limit, but for all values of $n$? The answer to this question is affirmative.

\begin{thm}\label{th3} Let $X_1, \ldots ,X_n$ be i.i.d. random variables taking values in interval $(0,1)$. Suppose additionally that:
\begin{enumerate} 
\item Distribution function $F(x)$ of $X_1$ is absolute continuous, strictly monotone on $(0,1)$, and $p(x)$ is its density.
\item $p(x)$ is regularly varying function of index $\alpha -1$ at zero.
\item $p(x)$ is differentiable on $(0,1)$ and the function
\[  \frac{d}{d x}\frac{p(x)}{\rho p(\rho x)} \]
is integrable on $(0,1)$.
\end{enumerate}	
Then the equality
\begin{equation}\label{eq14}
\p\{X_{(2)} \leq \rho X_{(1)}\} = 1-1/\rho^{\alpha}
\end{equation}	 
holds for all positive integer $n$ and all $\rho>1$ if and only if $1/X_1$ has Pareto distribution with parameter $\alpha$ and initial point $1$.
\end{thm}
\begin{proof} 
Let us suppose that $1/X_1$ has Pareto distribution with parameter $\alpha$ and initial point $1$. Then 
\[F(x) = x^{\alpha}\]
for $x \in (0,1)$. In this case, $p(x) = \alpha x^{\alpha -1}$ for $x \in (0,1)$ and $p(x) = 0$ otherwise. It is easy to calculate that $\frac{d}{d x}\frac{p(x)}{\rho p(\rho x)} = 0$. From (\ref{eq11}) it follows (\ref{eq14}).

Suppose now that (\ref{eq14}) holds. From (\ref{eq11}) we see that necessarily
\[ \int_{0}^{1/\rho}\left(F(\rho x)\right)^{n} \frac{d}{d x}\frac{p(x)}{\rho p(\rho x)} dx = 0  \]
for all positive integers $n$. In view of compactness of the interval $(0,1/\rho)$ and strictly monotone character of $F$ the problem of moments has unique solution. Therefore,
\[ \frac{d}{d x}\frac{p(x)}{\rho p(\rho x)} = 0 \]
for all $x \in (0,1/\rho)$. This implies that
\[ \frac{p(x)}{\rho p(\rho x)} = A(\rho), \]
where $A(\rho)$ depends on $\rho$ only. In other words, we have the following equation
\begin{equation}\label{eq15}
 p(x) = B(\rho)p(\rho x),
\end{equation}
for all $\rho >1$ and all $x \in (0,1/\rho)$, and $B(\rho)=\rho A(\rho)$. Passing to logarithms transforms (\ref{eq15}) to well-known Cauchy functional equation, which leads to $p(x) = \alpha x^{\alpha -1}$. 
\end{proof}

Let now $Y$ be a symmetric random variable such that $1/|Y|$ has Pareto distribution with parameter $\alpha$ and initial point $1$. We expect that $Y$ has outliers of the first kind with rather high probability although its distribution is not too close to extremal one and has a compact support.
Really, simple calculations give us that
\[\E Y =0, \;\; \sigma^2=\E Y^2 = \frac{\alpha}{2+\alpha}, \;\; \p\{|Y| \geq 3 \sigma\} =1-3^{\alpha}\Bigl(\frac{\alpha}{2+\alpha}\Bigr)^{\alpha/2}.  \]
For $ \alpha = 0.089115$ this probability is approximately $0.0417598$, which is greater than for Gaussian distribution.

As a conclusion of this section, we can say the presence of outliers of the first kind is not connected to tails of a distribution. It is associated with the behavior of the density near mean value. 

\section{Second definition of distant outliers}\label{s3}
\setcounter{equation}{0}
\subsection{Outliers of the second kind}\label{3ss1}

Here we are considering another look on distant outliers which was proposed in \cite{KKK}.  Namely, outlier in this sense is an extremal observation which is larger in its absolute value than $1/\kappa$ times previous extremal observation. Very similar definition may be founded in \cite{Haw}.

Let us give precise definition.
\begin{df}\label{de3}
Let $X_1, \ldots ,X_n$ be i.i.d. random variables, and $X_{(1)}, \ldots , X_{(n)}$ be the observations ordered in its absolute values (from minimal to maximal). We say $X_{(n)}$ is an outlier of order $1/\kappa$ if $X_{(n-1)}\leq \kappa X_{(n)}$, where $\kappa \in (0,1)$ is a fixed number.
\end{df} 
In this section we find a boundary for probability of outlier of order $1/\kappa$ and show its connection with the index of stability.
 
Let us calculate the probability of $X_{(n)}$ to be an outlier of order $1/\kappa$ for a fixed $\kappa \in (0,1)$. Suppose that $X_1, \ldots ,X_n$ are i.i.d. random variables, and $X_{(1)}, \ldots , X_{(n)}$ are the observations ordered in its absolute values. Suppose that random variable $|X_1|$ has absolute continuous distribution function $F(x)$, and $p(x)$ is its density. Denote by $p_{n-1,n}(x,y)$ the common density of $X_{(n-1)}$ and $X_{(n)}$. We have (see, for example, \cite{Dav}) 
\begin{equation}\label{eq16}
p_{n-1,n}(x,y) = n(n-1) F^{n-2}(x)p(x) p(y),
\end{equation}
for $x \leq y$. Therefore the probability of the event that $X_{(n-1)} \leq \kappa X_{(n)}$ is
\begin{equation}\label{eq17}
\p\{X_{(n-1)} \leq \kappa X_{(n)}\} = n \int_{0}^{\infty}F^{n-1}(\kappa y)p(y)dy.
\end{equation}
Let us try to study limit behavior of the probability (\ref{eq17}) for large values of sample size $n$. We have
\begin{equation}\label{eq18}
\p\{X_{(n-1)} \leq \kappa X_{(n)}\}=n \int_{0}^{\infty}F^{n-1}(x)\kappa p(\kappa x)\frac{p(x)}{\kappa p(\kappa x)}dx. 
\end{equation}
Assume that 
\begin{equation}\label{eq19}
\lim_{x \to 0}F^{n}(\kappa x)\frac{p(x)}{p(\kappa x)} =0. 
\end{equation}
Integrating by parts in (\ref{eq18}) gives us
\begin{equation}\label{eq20}
\begin{aligned}
\p\{X_{(n-1)} \leq \kappa X_{(n)}\}= \lim_{x \to \infty}\frac{p(x)}{\kappa p(\kappa x)}-\\ -\int_{0}^{\infty}F^{n}(\kappa x)\Bigl(\frac{p^{\prime}(x)}{\kappa p(\kappa x)}- \frac{p(x) p^{\prime}(\kappa x)}{p^2(\kappa x)} \Bigr) dx
\end{aligned}
\end{equation}
If the function
\[ \Bigl(\frac{p^{\prime}(x)}{\kappa p(\kappa x)}- \frac{p(x) p^{\prime}(\kappa x)}{p^2(\kappa x)} \Bigr) \]
is integrable over $(0,\infty)$ then 
\[ \int_{0}^{\infty}F^{n}(\kappa x)\Bigl(\frac{p^{\prime}(x)}{\kappa p(\kappa x)}- \frac{p(x) p^{\prime}(\kappa x)}{p^2(\kappa x)} \Bigr) dx \to 0 \]
as $n \to \infty$.
Therefore,
\begin{equation}\label{eq21}
\lim_{n \to \infty}\p\{X_{(n-1)} \leq \kappa X_{(n)}\}= \lim_{x \to \infty}\frac{p(x)}{\kappa p(\kappa x)},
\end{equation}
assuming that the limit in right-hand side of (\ref{eq21}) exists.

Finally, we obtain the following result.
\begin{thm}\label{th4}
Suppose that $X_1, \ldots ,X_n$ are i.i.d. random variables, and $X_{(1)}, \ldots , X_{(n)}$ are the observations ordered in its absolute values. Let random variable $|X_1|$ has absolute continuous distribution function $F(x)$, and let $p(x)$ be its density. Suppose that $p(x)$ is regularly varying function of index $-(\alpha+1)$ on infinity, the function 
\[ \Bigl(\frac{p^{\prime}(x)}{\kappa p(\kappa x)}- \frac{p(x) p^{\prime}(\kappa x)}{p^2(\kappa x)} \Bigr) \]
is integrable over $(0,\infty)$, and 
\[ \lim_{x \to 0}F^{n}(\kappa x)\frac{p(x)}{p(\kappa x)} =0.  \]
Then
\begin{equation}\label{eq22}
\lim_{n \to \infty}\p\{X_{(n-1)} \leq \kappa X_{(n)}\}=\kappa^{\alpha}.
\end{equation}
\end{thm}
\begin{proof}
The statement of the Theorem follows from considerations given above and from the definition of regularly varying function (see, for example \cite{Sen}).
\end{proof}

\subsection{Connection to the law of large numbers and statistical definition of stability index}\label{3ss2}

Theorem \ref{th4} shows that there is a connection between stable distribution and the probability of presence of $1/\kappa$ outliers.
Namely, the condition ``$p(x)$ is regularly varying function of index $-(\alpha+1)$ on infinity" implies that corresponding random variables $X_1, \ldots ,X_n$ belong to the region of attraction of $\alpha$-stable distribution. The probability (\ref{eq22}) is defined by index $\alpha$ in unique way, and increase with decreasing $\alpha$. 

For the first glance, it is not clear why there is no law of large numbers in the case of $\alpha \in (0,1)$. Really, in the case of symmetric distributions, it seems to be possible, that large positive observations may be compensated by corresponding negative observations, coming into empirical mean with the same probability as positive. Mean value of a mass distribution is a coordinate of the center of masses. One more argument for symmetric about zero distributions is that the mean value may be does not exist, but corresponding integral converges in Cauchy principal value. Therefore we may interpret the origin as corresponding center of masses. However, for $\alpha \in (0,1)$ the limit probability for $X_{(n-1)}$ to be less that $\kappa X_{(n)}$ is greater than $\kappa$ itself. It shows, that very often the ``maximal" observation $X_{(n)}$ cannot be ``compensated" by smaller observations. It gives us an intuitive explanation of why there is no law of large numbers for the case of $\alpha \in (0,1)$. 

Is it possible to use the relation (\ref{eq22}) to define the stability index $\alpha$? Of course, it is possible theoretically, but is impossible statistically, because we cannot pass to limit for any large (but finite) number $n$ of observations. However, the probability $\p\{X_{(n-1)}< \kappa X_{(n)}\}$ (for fixed $\kappa$ and $n$) may be statistically estimated. Such probability does not define ``true" value of $\alpha$, however, small value of such estimator for $\alpha$ shows that empirical mean is not close to any constant at least for corresponding values of $n$.

\subsection{Characterization of Pareto distribution}\label{3ss3}

\begin{thm}\label{th5} Let $X_1, \ldots ,X_n$ be i.i.d. random variables taking values in interval $(1,\infty)$. Suppose additionally that:
	\begin{enumerate} 
		\item Distribution function $F(x)$ of $X_1$ is absolute continuous, strictly monotone on $(1,\infty)$, and $p(x)$ is its density.
		\item $p(x)$ is regularly varying function of index $-(\alpha +1)$ at infinity.
		\item $p(x)$ is differentiable on $(1,\infty)$ and the function
		\[  \frac{d}{d x}\frac{p(x)}{\kappa p(\kappa x)} \]
		is integrable on $(1,\infty)$.
	\end{enumerate}	
	Then the equality
	\begin{equation}\label{eq14}
	\p\{X_{(n-1)} \leq \kappa X_{(n)}\} = \kappa ^{\alpha}
	\end{equation}	 
	holds for all positive integer $n$ and all $\kappa \in (0,1)$ if and only if $X_1$ has Pareto distribution with parameter $\alpha$ and initial point $1$.
\end{thm}
The proof of Theorem \ref{th5} uses the same ideas and similar calculations as that of Theorem \ref{th3} and is omitted.

\section{Outliers and multi-modality}\label{s4}
\setcounter{equation}{0}
The presence of two or more modes for empirical distribution contradicts to many parametric models. Such are, for example, models based on Gaussian or stable distributions. However, to understand how many modes has an empirical distribution one need to construct an non-parametric estimator for the density. To this aim it is necessarily to have a large number of observations. 

In this section we propose another theoretical approach to define outliers of such (third) kind. Namely, we propose to consider this as a multiple variants of the first kind outliers. Suppose that $X$ is a random variable. There must be some points $a_1, a_2, \ldots , a_k$ such that $|X-a_j|$ has outliers of the first kind for each $j=1,2, \ldots ,k$. In other words, the density $p(x)$ of random variable $X$ must have $k$ points, in which $p(x)$ is regularly varying function with different indexes. To see this, one may apply the methods of Section \ref{s2} to each random variable $|X-a_j|$. We omit other details.  

\section{Outliers in multivariate case}\label{s5}
\setcounter{equation}{0}

It is clear that there are much more possibilities for appearance of outliers in multidimensional case than in one dimensional. Unfortunately, we can not consider any large enough set of them. However, it is possible to mention some cases closely connected to one dimensional variant. 

The first (and more essential) case is the convex hull of sample points. The volume of this hull is one dimensional random variable. One may apply previously introduced definitions of outliers to this variable. The existence of outliers for the volume means that there are contradictions in multidimensional model as well. 

The second example is given by the distances (say, Euclidean) between sample points. The situation here is absolutely similar to the first example. It is, essentially, one dimensional, too. 

\section{Conclusions}

There were given some precise definitions of outliers.
It appears that the outliers of the first kind are connected to the presence of high pikes of the density, while second type outliers are associated with heavy tails of the distribution. Some definitions of outliers in multidimensional cases may be reduced to one dimensional case through the choice of appropriate characteristic of random vectors. The presence of outliers allows one to reject some parametric models. It provides also some ideas on how to construct properly modified models.

\section*{Acknowledgment}
The work was partially supported by Grant GA\v{C}R 16-03708S.


\begin{thebibliography}{99}
\bibitem{Kl16}
Lev B. Klebanov (2016).
\newblock Big Outliers Versus Heavy Tails: what to use?
\newblock arXiv 1611.05410v1, 1-14.

\bibitem{EK}
Erns Eberlein and Ulrich Keller (1995).
\newblock Hyperbolic Distributions in Finance,
\newblock Institut f\"{u}r Mathematische Stochastik, Universit\"{a}t Freiburg, 1-24.

\bibitem{BMW}
Szymon Borak, Adam Misiorek, Rafal Weron (2010).
\newblock Models for Heavy-tailed  Asset Returns,
\newblock SFB 649 Discussion Paper 2010-049, http://sfb649.wiwi.hu-berlin.de
\newblock ISSN 1860-5664 SFB 649, 
\newblock Humboldt-Universit\"{a}t zu Berlin, Spandauer Stra{\ss}e 1, D-10178 Berlin, 1-40.

\bibitem{KV}
Lev B. Klebanov, Irina Volchenkova (2015).
\newblock Heavy Tailed Distributions in Finance: Reality or Myth? Amateurs Viewpoint.
\newblock arXiv 1507.07735v1, 1-17.

\bibitem{KlF}
Lev B. Klebanov (2016).
\newblock No Stable Distributions in Finance, please!
\newblock arXiv 1601.00566v2, 1-9.

\bibitem{KTK}
Lev B Klebanov, Gregory Temnov, Ashot V. Kakosyan (2016).
\newblock Some Contra-Arguments for the Use of Stable Distributions in Financial Modeling,
\newblock arXiv 1602.00256v1, 1-9.

\bibitem{KS}
Samuel Karlin and William Studden (1966).
\newblock Tchebycheff Systems: With Applications in Analysis and Statistics, 
\newblock Interscience Publishers.

\bibitem{Dav}
H.A. David, H.N. Nagaraja (2003).
\newblock Order Statistics,
\newblock John Wiley \& Sons.

\bibitem{Sen}
Eugene Seneta (1976).
\newblock Regularly Varying Functions,
\newblock Springer, Berlin - Heidelberg.

\bibitem{KKK}
Lev B. Klebanov, Ashot V. Kakosyan, and Andrea Karlova (2016).
\newblock Outliers, the Law of Large Numbers, Index of
Stability and Heavy Tails,
\newblock arXiv 1612.09265v1, 1-5.

\bibitem{Haw}
D.M. Hawkins (1980).
\newblock Identification of outliers.
\newblock SPRINGER-SCIENCE+BUSINESS MEDIA, B.V.
\end{thebibliography}
\end{document}